\documentclass{amsart}

\usepackage[utf8]{inputenc}
\usepackage{latexsym,amsfonts,amssymb,amsthm,amsmath}
\usepackage{tikz}
\usetikzlibrary{arrows.meta}
\usepackage{enumitem}
\usetikzlibrary{arrows}
\usepackage{hyperref}

\newcommand{\RR}{\mathbb{R}}
\newcommand{\arrow}[1]{-{Stealth[length=3.5mm,line width=1pt, open,#1]}}

\numberwithin{figure}{section}

\newtheorem{theorem}{Theorem}[section]
\newtheorem{lemma}[theorem]{Lemma}
\newtheorem{definition}[theorem]{Definition}
\newtheorem{conjecture}[theorem]{Conjecture}
\newtheorem*{remark}{Remark}
\newtheorem{question}[theorem]{Question}
\newtheorem*{mainthm}{Main Theorem}

\tikzset{
	c/.style={every coordinate/.try}
}

\title[Towards Constructing Geodesic Nets with Four Boundary Vertices]{Towards Constructing Geodesic Nets with Four Boundary Vertices and an Increasing Number of Balanced Vertices}
\author{Fabian Parsch, Hanrui Zhang}
\address{Department of Mathematics, University of Toronto, 40 St. George Street, Toronto, ON M5S 2E4, Canada}
\email{fabian.parsch@utoronto.ca, hanrui.zhang@mail.utoronto.ca}
\copyrightinfo{2025}{Fabian Parsch, Hanrui Zhang}

\begin{document}

\begin{abstract}
    We construct a geodesic net in the plane with four boundary (unbalanced) vertices that has 25 balanced vertices and that is irreducible, i.e. it does not contain nontrivial subnets.
    
    This net is novel and remarkable for several reasons: (1) It increases the previously known maximum for balanced vertices of nets of this kind from 16 to 25. (2) It is, to our knowledge, the first such net that includes balanced vertices whose incident edges are not exhibiting symmetries of any kind. (3) The approach taken in the construction is quite promising as it might have the potential for generalization. This would allow to construct a series of irreducible geodesic nets with four boundary vertices and an arbitrary number of balanced vertices, answering a conjecture that the number of balanced vertices is in fact unbounded for nets with four boundary vertices. This would stand in stark contrast to the previously proven theorem that for three boundary vertices, there can be at most one single balanced vertex.
\end{abstract}

\maketitle

\section{Introduction}
\label{sec:intro}

Given two points in a Riemannian manifold, any geodesic segment connecting the two is a critical point of the length functional on the space of curves connecting them. We can generalize this approach to consider embedded graphs connecting a given set of three or more \textit{boundary points}. The critical points of the length functional on such a space of embedded graphs are geodesic nets. For a graph to be a geodesic net, every edge must be a geodesic segment. Additionally, except for the boundary vertices (also called \textit{unbalanced vertices}), all other vertices of a geodesic net must be \textit{balanced vertices}, as described in the following combinatorial definition.

\begin{definition}\label{def:geodesicnet}
	Let $S$ be a finite set of points in a Riemannian manifold $M$. Given a connected, finite graph $G=(V,E)$ embedded in $M$ and a subset of vertices $S\subset V$, we call $G$ a \emph{geodesic net with boundary vertex set $S$} if:
	\begin{enumerate}[label=({\roman*}),nosep]
		\item Every edge in $E$ is a geodesic segment.
		\item At each of the remaining vertices $v\in V\setminus S$ (the ``balanced vertices'') the following balancing condition holds: The sum of all unit vectors in the tangent space $T_vM$ directed along all edges from $v$ to the opposite end of each edge is equal to zero.
	\end{enumerate}
\end{definition}

As such, geodesic nets are a generalization of geodesic segments (which can be viewed as one-edge geodesic nets), and also a 1-dimensional analog of minimal surfaces. Their study, both regarding classification and regarding quantitative questions, is an area with a wealth of unanswered questions and open conjectures (see \cite{Nabutovsky2019-iy} for an overview of several of them). One of these questions is due to Gromov (see \cite{Gromov2009-fh}, p.799, where he considers a more general concept of \textit{edge-extremal graphs}).

\begin{question}[M. Gromov]\label{question:gromov}
Can the number of balanced vertices of a geodesic net in the Euclidean plane be bounded above in terms of the number of unbalanced vertices?
\end{question}

Note that this question naturally assumes that we do not allow degree two balanced vertices, as such vertices could be added and removed along any geodesic segment at will. Furthermore, we do not allow weighted edges with integer multiplicity (see \cite{Nabutovsky2019-iy} for comparisons between geodesic nets and geodesic multinets), and we are focusing on connected geodesic nets.

While we will focus on the Euclidean plane further below, it is worth noting what happens in the context of other geometries. To that end, consider the following more general version of Question \ref{question:gromov}.

\begin{question}\label{question:geometries}
    For certain geometries, can the number of balanced vertices of a geodesic net be bounded above in terms of the number of unbalanced vertices?
\end{question}

One archetypical geometry to consider Question \ref{question:geometries} in is the round 2-sphere. In this case, the answer is ``no'' and rather trivially so: Combining an arbitrary number of great circles leads to geodesic nets with no unbalanced but an arbitrarily large number of balanced vertices (at the intersections of the great circles). On the other hand, there are other fascinating questions about the possible shapes of geodesic nets on the round sphere. Questions of existence of geodesic nets with certain shapes are considered in \cite{Hass1996-ov} and \cite{Heppes1999-oh}. In a more general setup, \cite{Adelstein2020-fc} considers bounds on the number of balanced vertices of certain geodesic nets in manifolds homeomorphic to the n-sphere.

Another archetype is the plane endowed with a Riemannian metric. If we require that the curvature is everywhere nonpositive, it is immediate that nets with zero, one or two unbalanced vertices have no balanced vertices (in the case of two unbalanced vertices, the only option is a geodesic segment connecting the two since we do not allow degree two balanced vertices). This makes the case of three unbalanced vertices the first of particular interest.

If three boundary vertices in the flat plane form a triangle with all interior angles less than $2\pi/3$, then we can add their Fermat point to get a \tikz[scale=0.3]{\draw(0,0)--(10:0.5);\draw(0,0)--(130:0.5);\draw(0,0)--(250:0.5);} shaped gedoesic net with a single balanced vertex. One set of conditions under which such a ``generalized Fermat point'' exists in non-flat metrics is considered in \cite{Nguyen2025-wu}.

As it turns out, one such balanced point is the maximal number for three boundary points on a plane with nonpositive curvature, as was shown in \cite{Parsch2018-nu}. In other words, in this particular situation the answer to Question \ref{question:geometries} is ``yes''.

\begin{theorem}\label{thm:three}
	Each geodesic net with three unbalanced vertices (of arbitrary degree) on the plane endowed with a Riemannian metric of non-positive curvature has at most one balanced vertex. (Furthermore, this statement is not true if we allow for positive curvature.)
\end{theorem}

The natural follow-up scenario are geodesic nets with four boundary vertices. Very little is known with regards to the above question, even in the Euclidean plane. We will now focus on that geometry, i.e. we return to Question  \ref{question:gromov}.

As an example, consider the geodesic net shown on the left of Figure \ref{fig:geonet_4_27_reducible}, as constructed in \cite{Parsch2018-nu}. While it does demonstrate that a quite significant number of balanced vertices are possible, it isn't particularly complex as it is an overlay of seven simpler tree-shaped geodesic nets, as shown on the right of the figure.

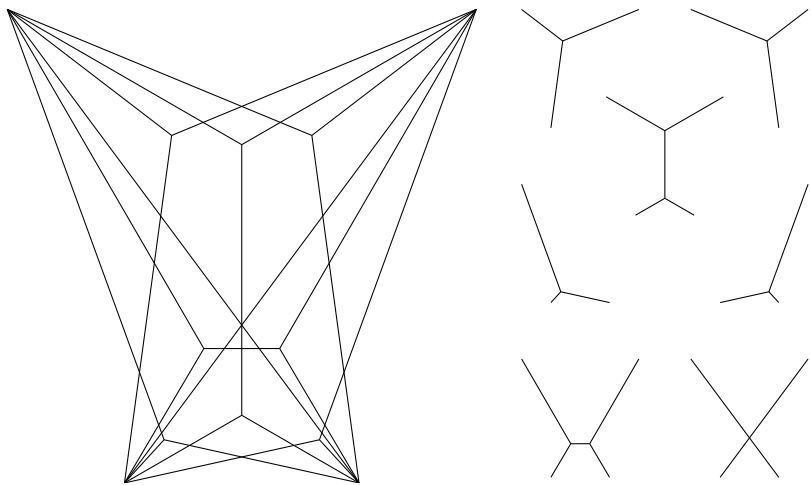
\begin{figure}[t]
  \centering
  
  \begin{tikzpicture}[scale=0.6]
	\clip (-5.5,1.5) rectangle (13,12);
	\node at (0,0) (Q) {};
	\node at (150:3) (X) {};
	\node at (125:3) (Y1) {};
	\node at (90:3) (Y2) {};
	\node at (55:3) (Y3) {};
	\node at (30:3) (Z) {};
	
	\node at (0,15) (P) {};
	\path (P) -- ++(210:6) node (A) {};
	\path (P) -- ++(255:6) node (B1) {};
	\path (P) -- ++(270:6) node (B2) {};
	\path (P) -- ++(285:6) node (B3) {};
	\path (P) -- ++(330:6) node (C) {};
	
	\node at (-.84,4.48) (L) {};
	\node at (.84,4.48) (N) {};

	\draw (A.center) -- (B1.center) -- (C.center);
	\draw (A.center) -- (B2.center) -- (C.center);
	\draw (A.center) -- (B3.center) -- (C.center);
	
	\draw (X.center) -- (Y1.center) -- (Z.center);
	\draw (X.center) -- (Y2.center) -- (Z.center);
	\draw (X.center) -- (Y3.center) -- (Z.center);
	
	\draw (B1.center) -- (X.center);
	\draw (B3.center) -- (Z.center);
	
	\draw (Y1.center) -- (A.center);
	\draw (Y3.center) -- (C.center);
	
	\draw (B2.center) -- (Y2.center);
	
	\draw (A.center) -- (L.center);
	\draw (X.center) -- (L.center);
	\draw (Z.center) -- (N.center);
	\draw (C.center) -- (N.center);
	\draw (L.center) -- (N.center);
	
	\draw (A.center) -- (Z.center);
	\draw (C.center) -- (X.center);
	\begin{scope}[scale=0.25, yshift=36cm, xshift=30cm]
	\node at (0,0) (Q) {};
	\node at (150:3) (X) {};
	\node at (125:3) (Y1) {};
	\node at (90:3) (Y2) {};
	\node at (55:3) (Y3) {};
	\node at (30:3) (Z) {};
	
	\node at (0,15) (P) {};
	\path (P) -- ++(210:6) node (A) {};
	\path (P) -- ++(255:6) node (B1) {};
	\path (P) -- ++(270:6) node (B2) {};
	\path (P) -- ++(285:6) node (B3) {};
	\path (P) -- ++(330:6) node (C) {};
	
	\node at (-.84,4.48) (L) {};
	\node at (.84,4.48) (N) {};

	\draw (A.center) -- (B1.center) -- (C.center);
	
	\draw (B1.center) -- (X.center);
	\end{scope}
	\begin{scope}[scale=0.25, yshift=36cm, xshift=45cm]
	\node at (0,0) (Q) {};
	\node at (150:3) (X) {};
	\node at (125:3) (Y1) {};
	\node at (90:3) (Y2) {};
	\node at (55:3) (Y3) {};
	\node at (30:3) (Z) {};
	
	\node at (0,15) (P) {};
	\path (P) -- ++(210:6) node (A) {};
	\path (P) -- ++(255:6) node (B1) {};
	\path (P) -- ++(270:6) node (B2) {};
	\path (P) -- ++(285:6) node (B3) {};
	\path (P) -- ++(330:6) node (C) {};
	
	\node at (-.84,4.48) (L) {};
	\node at (.84,4.48) (N) {};

	\draw (A.center) -- (B3.center) -- (C.center);
	
	\draw (B3.center) -- (Z.center);
	\end{scope}
	\begin{scope}[scale=0.25, yshift=28.25cm, xshift=37.5cm]
	\node at (0,0) (Q) {};
	\node at (150:3) (X) {};
	\node at (125:3) (Y1) {};
	\node at (90:3) (Y2) {};
	\node at (55:3) (Y3) {};
	\node at (30:3) (Z) {};
	
	\node at (0,15) (P) {};
	\path (P) -- ++(210:6) node (A) {};
	\path (P) -- ++(255:6) node (B1) {};
	\path (P) -- ++(270:6) node (B2) {};
	\path (P) -- ++(285:6) node (B3) {};
	\path (P) -- ++(330:6) node (C) {};
	
	\node at (-.84,4.48) (L) {};
	\node at (.84,4.48) (N) {};

	\draw (A.center) -- (B2.center) -- (C.center);
	
	\draw (X.center) -- (Y2.center) -- (Z.center);
	
	\draw (B2.center) -- (Y2.center);
	\end{scope}
	
	\begin{scope}[scale=0.25, yshift=20.5cm, xshift=30cm]
	\node at (0,0) (Q) {};
	\node at (150:3) (X) {};
	\node at (125:3) (Y1) {};
	\node at (90:3) (Y2) {};
	\node at (55:3) (Y3) {};
	\node at (30:3) (Z) {};
	
	\node at (0,15) (P) {};
	\path (P) -- ++(210:6) node (A) {};
	\path (P) -- ++(255:6) node (B1) {};
	\path (P) -- ++(270:6) node (B2) {};
	\path (P) -- ++(285:6) node (B3) {};
	\path (P) -- ++(330:6) node (C) {};
	
	\node at (-.84,4.48) (L) {};
	\node at (.84,4.48) (N) {};

	\draw (X.center) -- (Y1.center) -- (Z.center);
	
	\draw (Y1.center) -- (A.center);
	\end{scope}
	\begin{scope}[scale=0.25, yshift=20.5cm, xshift=45cm]
	\node at (0,0) (Q) {};
	\node at (150:3) (X) {};
	\node at (125:3) (Y1) {};
	\node at (90:3) (Y2) {};
	\node at (55:3) (Y3) {};
	\node at (30:3) (Z) {};
	
	\node at (0,15) (P) {};
	\path (P) -- ++(210:6) node (A) {};
	\path (P) -- ++(255:6) node (B1) {};
	\path (P) -- ++(270:6) node (B2) {};
	\path (P) -- ++(285:6) node (B3) {};
	\path (P) -- ++(330:6) node (C) {};
	
	\node at (-.84,4.48) (L) {};
	\node at (.84,4.48) (N) {};

	\draw (X.center) -- (Y3.center) -- (Z.center);
	
	\draw (Y3.center) -- (C.center);
	\end{scope}
	
	\begin{scope}[scale=0.25, yshift=5cm, xshift=30cm]
	\node at (0,0) (Q) {};
	\node at (150:3) (X) {};
	\node at (125:3) (Y1) {};
	\node at (90:3) (Y2) {};
	\node at (55:3) (Y3) {};
	\node at (30:3) (Z) {};
	
	\node at (0,15) (P) {};
	\path (P) -- ++(210:6) node (A) {};
	\path (P) -- ++(255:6) node (B1) {};
	\path (P) -- ++(270:6) node (B2) {};
	\path (P) -- ++(285:6) node (B3) {};
	\path (P) -- ++(330:6) node (C) {};
	
	\node at (-.84,4.48) (L) {};
	\node at (.84,4.48) (N) {};

	\draw (A.center) -- (L.center);
	\draw (X.center) -- (L.center);
	\draw (Z.center) -- (N.center);
	\draw (C.center) -- (N.center);
	\draw (L.center) -- (N.center);
	\end{scope}
	\begin{scope}[scale=0.25, yshift=5cm, xshift=45cm]
	\node at (0,0) (Q) {};
	\node at (150:3) (X) {};
	\node at (125:3) (Y1) {};
	\node at (90:3) (Y2) {};
	\node at (55:3) (Y3) {};
	\node at (30:3) (Z) {};
	
	\node at (0,15) (P) {};
	\path (P) -- ++(210:6) node (A) {};
	\path (P) -- ++(255:6) node (B1) {};
	\path (P) -- ++(270:6) node (B2) {};
	\path (P) -- ++(285:6) node (B3) {};
	\path (P) -- ++(330:6) node (C) {};
		
	\draw (A.center) -- (Z.center);
	\draw (C.center) -- (X.center);	
	\end{scope}

\end{tikzpicture}

\caption{Example of a geodesic net with 4 boundary vertices and 27 balanced vertices on the Euclidean plane with the boundary vertices highlighted, as constructed in Section 7 of \cite{Parsch2018-nu}. Despite its seemingly complex structure, it is just an overlay of 7 simpler nets as seen on the right.}
\label{fig:geonet_4_27_reducible}
\end{figure}

This observation motivated the definition of \textit{irreducible} geodesic nets in \cite{Parsch2019-fz}.

\begin{definition}
	\begin{enumerate}[label=({\roman*}),nosep]
		\item Given two geodesic nets $G'$ and $G$, we say that $G'$ is a \textit{subnet} of $G$ if $G'$ is a subgraph of $G$ and the boundary vertex set of $G'$ is a subset of the boundary vertex set of $G$.
		\item A geodesic net $G$ is \emph{irreducible} if for any geodesic net $G'$:
		\begin{align*}
			G'\text{ is a subnet of }G\quad\Longrightarrow\quad G'=G\quad\text{or}\quad G'\text{ consists of a single point}
		\end{align*}
	\end{enumerate}
\end{definition}
In other words, a geodesic net is irreducible if it has only trivial subnets. This means, for example, that the net on the left of Figure \ref{fig:geonet_4_27_reducible} is not irreducible, but each of the nets on the right is.

\cite{Parsch2019-fz} provides the construction of an irreducible geodesic net with four boundary vertices that isn't just a tree and has 16 balanced vertices (see the top of Figure \ref{fig:twoirreduciblenets}). It also includes the remark that it is ``tempting to conjecture that [the] example [constructed there] is one of a series of similar examples with arbitrary large number of balanced vertices''. This is alluding to Conjecture 3.2.2 in \cite{Nabutovsky2019-iy}, namely:

\begin{conjecture}\label{conj:fourandinfinite}
There exist geodesic nets in the Euclidean plane with 4 unbalanced vertices and an arbitrarily large number of balanced vertices. (Moreover, we will not be surprised if this assertion is already true in the case when the set of unbalanced vertices coincides with the set of vertices of a square).	
\end{conjecture}

In other words, this conjecture claims that the answer to Question \ref{question:gromov} is ``No'' for four or more unbalanced vertices\footnote{It is worth noting that \cite{Nabutovsky2019-iy} also mentions a different conjecture that it is possible to bound the number of balanced vertices in terms of the number of unbalanced vertices \textit{and} an additional quantity called the \textit{total imbalance}. For more details, including the definition of \textit{total imbalance}, see that article.}. This would stand in stark contrast to the bound for three unbalanced vertices in Theorem \ref{thm:three}.

Note that Conjecture \ref{conj:fourandinfinite} doesn't require the respective nets to be irreducible. But given that irreducible nets are the ``building blocks'' of any possible net, their study is particularly promising when it comes to finding such a series of nets.

Seeing that the example in \cite{Parsch2019-fz} has four boundary vertices and 16 balanced vertices, the following question is posed in Section 3.2 of \cite{Nabutovsky2019-iy} in the sense of a ``stepping stone'' to proving Conjecture \ref{conj:fourandinfinite}.

\begin{question}
Is there an irreducible geodesic net with 4 unbalanced vertices in the Euclidean plane with more than 16 balanced vertices? If not, prove that such a geodesic net does not exist.
\end{question}

The main purpose of the present paper is to show that the answer to this question is yes.

\begin{mainthm}
    There exists an irreducible geodesic net with 4 boundary vertices and 25 balanced vertices. 
\end{mainthm}

\begin{figure}[htbp]

	\flushright
 \begin{tikzpicture}[scale=1.4]	
		\draw (0:1) coordinate (a1) -- (45:1.12) coordinate (b1) -- (90:1) coordinate (a2) -- (135:1.12) coordinate (b2) -- (180:1) coordinate (a3) -- (225:1.12) coordinate (b3) -- (270:1) coordinate (a4) -- (315:1.12) coordinate (b4) -- cycle;
		
		\fill (0:{tan(76)}) coordinate (c1) circle (0.05);
		\fill (90:{tan(76)}) coordinate (c2) circle (0.05);
		\fill (180:{tan(76)}) coordinate (c3) circle (0.05);
		\fill (270:{tan(76)}) coordinate (c4) circle (0.05);
		
		\draw (a1) -- (c2) -- (a3) -- (c4) -- (a1);
		\draw (a2) -- (c3) -- (a4) -- (c1) -- (a2);
		
		\draw (a1) -- (c1);
		\draw (a2) -- (c2);
		\draw (a3) -- (c3);
		\draw (a4) -- (c4);
		
		\draw (b1) -- ++ (45:0.07) coordinate (d1);
		\draw (c1) -- (d1) -- (c2);
		
		\draw (b2) -- ++ (135:0.07) coordinate (d2);
		\draw (c2) -- (d2) -- (c3);
		
		\draw (b3) -- ++ (225:0.07) coordinate (d3);
		\draw (c3) -- (d3) -- (c4);
		
		\draw (b4) -- ++ (315:0.07) coordinate (d4);
		\draw (c4) -- (d4) -- (c1);
		
		\draw [dashed] (b1) circle (0.3);
		\draw [dashed] (a1) circle (0.3);
		
		\begin{scope}[every coordinate/.style={shift={(-5.5,-5.5)},scale=10}]
			\begin{scope}
			\clip ([c]b1) circle (1.2);
			\draw ([c]a1) -- ([c]b1) -- ([c]a2);
			\draw ([c]a1) -- ([c]c2);
			\draw ([c]a2) -- ([c]c1);
			\draw ([c]b1) -- ([c]d1);
			\draw ([c]c1) -- ([c]d1) -- ([c]c2);
			\end{scope}
			\draw [dashed] ([c]b1) circle (1.2);
		\end{scope}
		
		\begin{scope}[every coordinate/.style={shift={(-7.5,-2.5)},scale=10}]
			\begin{scope}
			\clip ([c]a1) circle (1.2);
			\draw ([c]a1) -- ++ (120:3);
			\draw ([c]a1) -- ++ (-120:3);
			\draw ([c]a1) -- ++ (105:3);
			\draw ([c]a1) -- ++ (-105:3);
			\draw ([c]a1) -- ([c]c1);
			\path ([c]a1) -- ++(0:0.5) coordinate (arcstart1);
			\path ([c]a1) -- ++(0:1) coordinate (arcstart2);
			\end{scope}
			\path ([c]a1) -- ++(315:0.65) node [text width=2cm,align=center] {\scriptsize some angles exaggerated\par};
			\draw [dashed] ([c]a1) circle (1.2);
		\end{scope}

	\end{tikzpicture}
	
	\flushleft
	\vspace{-30mm}
    \begin{tikzpicture}[scale=0.8]
\draw (0,0)-- (0.753334985772626,0);
\draw (0.753334985772626,0)-- (1.7192608120616921,0.2588190451025174);
\draw (1.7192608120616921,0.2588190451025174)-- (1.9780798571642149,1.2247448713915854);
\draw (1.9780798571642149,1.2247448713915854)-- (1.9780798571642162,1.9780798571642115);
\draw (1.9780798571642162,1.9780798571642115)-- (1.7192608120616968,2.9440056834532804);
\draw (1.7192608120616968,2.9440056834532804)-- (0.7533349857726285,3.2028247285558025);
\draw (0.7533349857726285,3.2028247285558025)-- (0,3.2028247285558034);
\draw (0,3.2028247285558034)-- (-0.9659258262890664,2.9440056834532835);
\draw (-0.9659258262890664,2.9440056834532835)-- (-1.224744871391588,1.9780798571642153);
\draw (-1.224744871391588,1.9780798571642153)-- (-1.2247448713915885,1.2247448713915894);
\draw (-1.2247448713915885,1.2247448713915894)-- (-0.9659258262890682,0.258819045102521);
\draw (-0.9659258262890682,0.258819045102521)-- (0,0);
\draw (-1.224744871391588,1.9780798571642153)-- (-1.0072757929786942,1.6014123642779021);
\draw (-1.2247448713915885,1.2247448713915894)-- (-1.0072757929786942,1.6014123642779021);
\draw (-1.0072757929786942,1.6014123642779021)-- (0.3766674928863143,1.6014123642779008);
\draw (-6.516205728898051,1.6014123642779023)-- (-1.224744871391588,1.9780798571642153);
\draw (-6.516205728898051,1.6014123642779023)-- (-1.2247448713915885,1.2247448713915894);
\draw (0,0)-- (0.376667492886313,-5.291460857506463);
\draw (0.753334985772626,0)-- (0.376667492886313,-5.291460857506463);
\draw (-0.9831319305486067,0.24161294084298168)-- (-0.9659258262890682,0.258819045102521);
\draw (-1.0799680129622857,0.14477685842930255)-- (-0.9831319305486067,0.24161294084298168);
\draw (0.376667492886313,-5.291460857506463)-- (-1.0799680129622857,0.14477685842930255);
\draw (0.376667492886313,-5.291460857506463)-- (-0.9831319305486067,0.24161294084298168);
\draw (-6.516205728898051,1.6014123642779023)-- (-1.0799680129622857,0.14477685842930255);
\draw (-6.516205728898051,1.6014123642779023)-- (-0.9831319305486067,0.24161294084298168);
\draw (-0.9831319305486067,0.24161294084298168)-- (-1.2247448713915885,1.2247448713915894);
\draw (0,0)-- (-0.9831319305486067,0.24161294084298168);
\draw (7.269540714670677,1.6014123642778986)-- (1.9780798571642162,1.9780798571642115);
\draw (7.269540714670677,1.6014123642778986)-- (1.9780798571642149,1.2247448713915854);
\draw (0.3766674928863153,8.494285586062267)-- (0,3.2028247285558034);
\draw (0.3766674928863153,8.494285586062267)-- (0.7533349857726285,3.2028247285558025);
\draw (0.3766674928863153,8.494285586062267)-- (-0.9831319305486066,2.9612117877128212);
\draw (-6.516205728898051,1.6014123642779023)-- (-0.9831319305486066,2.9612117877128212);
\draw (-1.224744871391588,1.9780798571642153)-- (-0.9831319305486066,2.9612117877128212);
\draw (-0.9831319305486066,2.9612117877128212)-- (-0.9659258262890664,2.9440056834532835);
\draw (-0.9831319305486066,2.9612117877128212)-- (0,3.2028247285558034);
\draw (-1.079968012962285,3.0580478701265017)-- (-0.9831319305486066,2.9612117877128212);
\draw (0.3766674928863153,8.494285586062267)-- (-1.079968012962285,3.0580478701265017);
\draw (-1.079968012962285,3.0580478701265017)-- (-6.516205728898051,1.6014123642779023);
\draw (1.7192608120616968,2.9440056834532804)-- (1.7364669163212336,2.961211787712821);
\draw (0.7533349857726285,3.2028247285558025)-- (1.7364669163212336,2.961211787712821);
\draw (0.3766674928863153,8.494285586062267)-- (1.7364669163212336,2.961211787712821);
\draw (1.7364669163212336,2.961211787712821)-- (1.9780798571642162,1.9780798571642115);
\draw (1.7364669163212336,2.961211787712821)-- (7.269540714670677,1.6014123642778986);
\draw (1.8333029987349136,3.0580478701265)-- (1.7364669163212336,2.961211787712821);
\draw (0.3766674928863153,8.494285586062267)-- (1.8333029987349136,3.0580478701265);
\draw (1.8333029987349136,3.0580478701265)-- (7.269540714670677,1.6014123642778986);
\draw (1.7364669163212336,0.24161294084298132)-- (1.7192608120616921,0.2588190451025174);
\draw (0.376667492886313,-5.291460857506463)-- (1.7364669163212336,0.24161294084298132);
\draw (1.7364669163212336,0.24161294084298132)-- (7.269540714670677,1.6014123642778986);
\draw (1.7364669163212336,0.24161294084298132)-- (1.9780798571642149,1.2247448713915854);
\draw (1.7364669163212336,0.24161294084298132)-- (0.753334985772626,0);
\draw (1.8333029987349128,0.14477685842930033)-- (1.7364669163212336,0.24161294084298132);
\draw (0.376667492886313,-5.291460857506463)-- (1.8333029987349128,0.14477685842930033);
\draw (1.8333029987349128,0.14477685842930033)-- (7.269540714670677,1.6014123642778986);
\draw (0,3.2028247285558034)-- (0.37666749288631535,2.9853556501429086);
\draw (0.37666749288631535,2.9853556501429086)-- (0.7533349857726285,3.2028247285558025);
\draw (0.37666749288631535,2.9853556501429086)-- (0.3766674928863143,1.6014123642779008);
\draw (1.9780798571642162,1.9780798571642115)-- (1.7606107787513212,1.6014123642778983);
\draw (1.7606107787513212,1.6014123642778983)-- (1.9780798571642149,1.2247448713915854);
\draw (1.7606107787513212,1.6014123642778983)-- (0.3766674928863143,1.6014123642779008);
\draw (0,0)-- (0.376667492886313,0.21746907841289428);
\draw (0.376667492886313,0.21746907841289428)-- (0.753334985772626,0);
\draw (0.376667492886313,0.21746907841289428)-- (0.3766674928863143,1.6014123642779008);
\draw [fill=black] (-6.516205728898051,1.6014123642779023) circle (2.5pt);
\draw [fill=black] (0.376667492886313,-5.291460857506463) circle (2.5pt);
\draw [fill=black] (7.269540714670677,1.6014123642778986) circle (2.5pt);
\draw [fill=black] (0.3766674928863153,8.494285586062267) circle (2.5pt);

\draw[dashed] (-0.96,2.94) circle (0.5);

\begin{scope}[xshift=-0.5cm,yshift=-7.5cm,scale=4.5]
\draw[dashed] (-0.96,2.94) circle (0.5);
\begin{scope}
\clip (-0.96,2.94) circle (0.5);
\draw (0,3.2028247285558034)-- (-0.9659258262890664,2.9440056834532835);
\draw (-0.9659258262890664,2.9440056834532835)-- (-1.224744871391588,1.9780798571642153);
\draw (0.3766674928863153,8.494285586062267)-- (-0.9831319305486066,2.9612117877128212);
\draw (-6.516205728898051,1.6014123642779023)-- (-0.9831319305486066,2.9612117877128212);
\draw (-1.224744871391588,1.9780798571642153)-- (-0.9831319305486066,2.9612117877128212);
\draw (-0.9831319305486066,2.9612117877128212)-- (-0.9659258262890664,2.9440056834532835);
\draw (-0.9831319305486066,2.9612117877128212)-- (0,3.2028247285558034);
\draw (-1.079968012962285,3.0580478701265017)-- (-0.9831319305486066,2.9612117877128212);
\draw (0.3766674928863153,8.494285586062267)-- (-1.079968012962285,3.0580478701265017);
\draw (-1.079968012962285,3.0580478701265017)-- (-6.516205728898051,1.6014123642779023);
\end{scope}
\end{scope}

\draw[dashed] (-1.22,1.60) circle (0.7);

\begin{scope}[xshift=0.8cm,yshift=-10cm,scale=4.5]
\draw[dashed] (-1.22,1.60) circle (0.5);
\begin{scope}
\clip (-1.22,1.60) circle (0.5);
\draw (-1.224744871391588,1.9780798571642153) -- ++(60:1);
\draw (-1.224744871391588,1.9780798571642153)-- (-1.2247448713915885,1.2247448713915894);
\draw (-1.2247448713915885,1.2247448713915894)-- ++(-60:1);
\draw (-1.224744871391588,1.9780798571642153)-- (-1.0072757929786942,1.6014123642779021);
\draw (-1.2247448713915885,1.2247448713915894)-- (-1.0072757929786942,1.6014123642779021);
\draw (-1.0072757929786942,1.6014123642779021)-- (0.3766674928863143,1.6014123642779008);
\draw (-6.516205728898051,1.6014123642779023)-- (-1.224744871391588,1.9780798571642153);
\draw (-6.516205728898051,1.6014123642779023)-- (-1.2247448713915885,1.2247448713915894);
\draw (-0.9831319305486067,0.24161294084298168)-- (-1.2247448713915885,1.2247448713915894);
\draw (-1.224744871391588,1.9780798571642153)-- (-0.9831319305486066,2.9612117877128212);
\node at (-1.45,1.60) [text width=2cm,align=center] {\scriptsize some angles exaggerated\par};
\end{scope}
\end{scope}
\end{tikzpicture}
    \caption{Two irreducible geodesic nets with 4 boundary vertices. At the top, the example with 16 balanced vertices, constructed in \cite{Parsch2019-fz}. At the bottom, the example with 25 balanced vertices that we construct in the present paper. Note that in both cases, some edges may seem to coincide due to angles being extremely small. This is why some angles are slightly exaggerated in the zoom-ins.}
    \label{fig:twoirreduciblenets}
\end{figure}
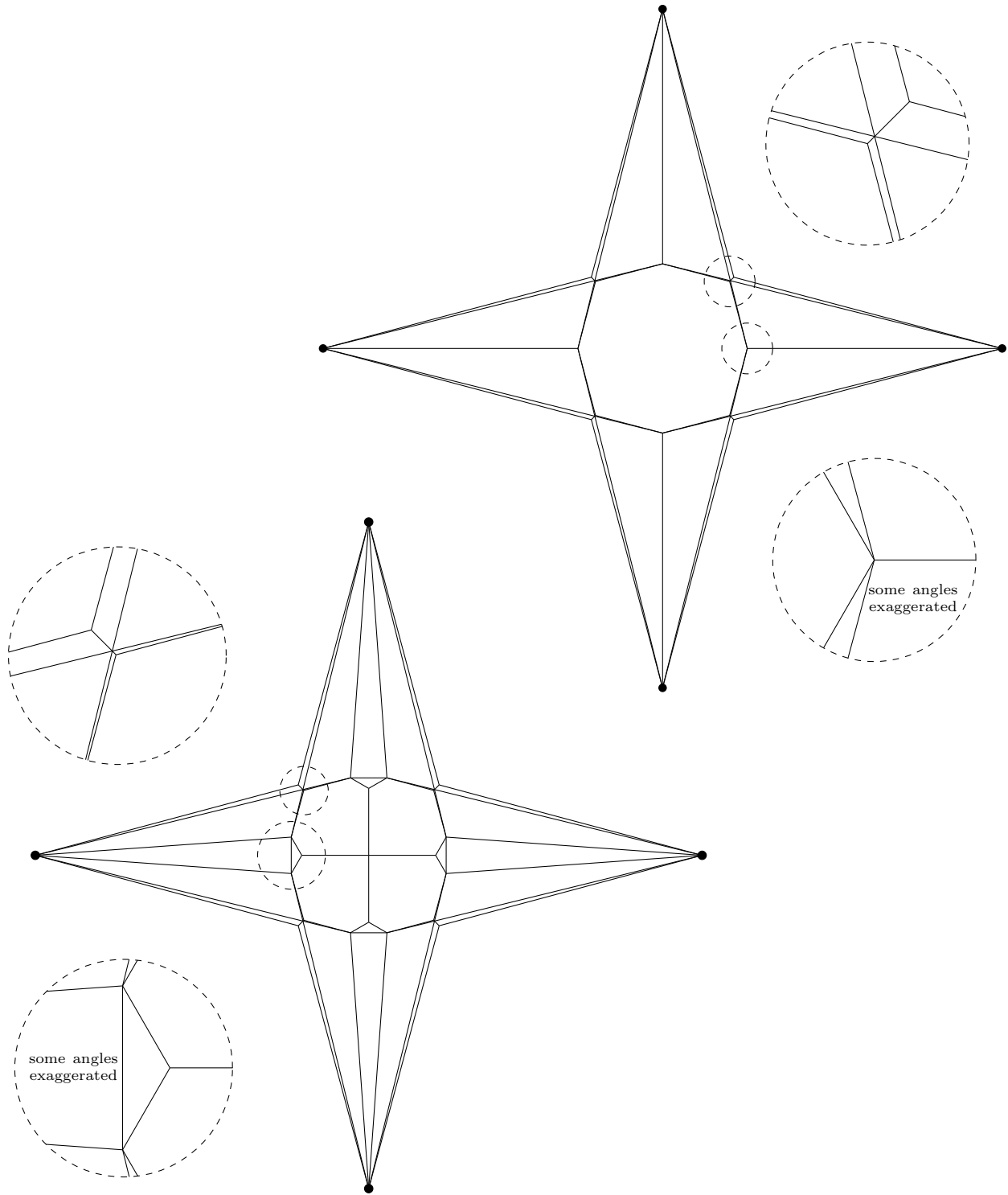

This geodesic net can be seen at the bottom of Figure \ref{fig:twoirreduciblenets}. We will provide its construction in Section \ref{section:construction}, and then prove that it fulfills the conditions of the Main Theorem. There are three important but technical lemmas that we will use in the process. For better readability of the main construction, their proofs are deferred and can be found separately in Section \ref{section:proof}.

By comparing the two nets in Figure \ref{fig:twoirreduciblenets}, namely the one previously constructed in \cite{Parsch2019-fz} and the one constructed below, it is not hard to notice that these nets have a very similar structure, as if they are the first two steps in a series of nets. We consider this the second noteworthy result of this paper.

More specifically, the genesis of the net constructed below provides interesting approaches and methods that might be useful in the quest for finding a series of geodesic nets with four boundary vertices and an ever increasing number of balanced vertices (i.e. an answer to Conjecture \ref{conj:fourandinfinite}). We will take a closer look at this in Section \ref{section:outro} in the hope that it might inform future research on the topic.

\section{Construction of the Net}\label{section:construction}

\subsection{Conventions}

The net will be symmetric under rotation by $\pi/2$, and also symmetric under reflection along the horizontal, vertical, and the two diagonal lines. We will make heavy use of these symmetries in the arguments, sometimes without explicit reference.

The construction of the net will use indices $i\in\{1,2,3,4\}$. In the spirit of the rotational symmetry, we will use cyclical indexing modulo 4, e.g. for an index $i=4$, we have $i+1=1$.

For two points $p,q\in\RR^2$, $\overline{pq}$ denotes the straight line segment between $p$ and $q$ whereas $d(p,q)$ denotes their Euclidean distance. Finally, $\angle pqr$ denotes the counterclockwise angle from $p$ to $r$ at $q$.

\subsection{Overview}

The geodesic net that is about to be constructed is shown in its entirety at the bottom of Figure \ref{fig:twoirreduciblenets}. All vertices except for the four labelled boundary vertices are balanced.

\subsection{The Two Angles \texorpdfstring{$\alpha$}{alpha} and \texorpdfstring{$\beta$}{beta}}

The construction of the net will use specific angles $\alpha$ and $\beta$. These two angles are chosen based on the following lemma.

\begin{lemma}
    \label{lemma:alpha_beta_existence}
    There is a unique solution $(\alpha,\beta)$ to the following system of equations under the constraint that $\alpha\in (\pi, \frac{13\pi}{12})$ and $\beta\in (0, \frac{\pi}{2})$:
    \begin{align*}
        1 + \cos \beta + \cos\alpha + \cos\frac{13\pi}{12} + \cos\frac{11\pi}{6} =&\; 0,\\
        \sin \beta + \sin\alpha + \sin\frac{13\pi}{12} + \sin\frac{11\pi}{6} =&\; 0,
    \end{align*}
\end{lemma}

The proof of this lemma is deferred to Section \ref{section:proof}.

\subsection{Construction of Inner Dodecagon:}\label{subsection:const_12gon}

To start, we will construct the unique dodecagon (up to rotation) with points $\{b_1, a_{11}, a_{12}, b_{2}, a_{21}, a_{22}, b_3, a_{31}, a_{32}, b_4, a_{41}, a_{42}\}$ as shown in Figure \ref{fig:12gon}, using these constraints:
\begin{itemize}
    \item The interior angle at each $b_i$ is $2\pi/3$.
    \item The interior angle at each $a_{i1}$ and $a_{i2}$  is $11\pi/12$.
    \item $d(b_i,a_{i1})=d(a_{i2},b_{i+1})=1$ (note the modulo convention above)
    \item $d(a_{i1},a_{i2})= \frac{\sqrt{6}\cdot (1 - \tan(\alpha))}{\tan(\alpha)\cdot \tan(\beta) - 1}$
\end{itemize}

\textit{For $\alpha$ and $\beta$ given by Lemma \ref{lemma:alpha_beta_existence}, we get $d(a_{i1},a_{i1})\approx 0.7533$. Naturally, the net could be scaled arbitrarily, as long as the ratio between the two distances used is maintained.}

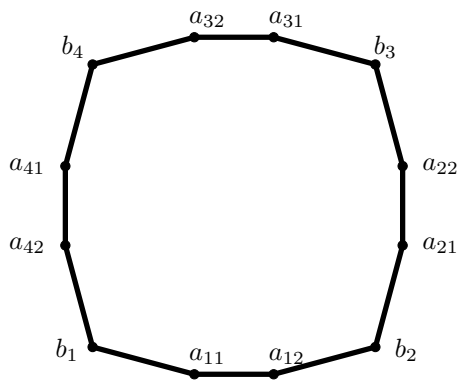
\begin{figure}[htbp]
\centering
\begin{tikzpicture}[line cap=round,line join=round,>=triangle 45,x=2cm,y=2cm, scale=0.7]
\draw [line width=2pt] (0,0)-- (0.753334985772626,0);
\draw [line width=2pt] (0.753334985772626,0)-- (1.7192608120616921,0.2588190451025174);
\draw [line width=2pt] (1.7192608120616921,0.2588190451025174)-- (1.9780798571642149,1.2247448713915854);
\draw [line width=2pt] (1.9780798571642149,1.2247448713915854)-- (1.9780798571642162,1.9780798571642115);
\draw [line width=2pt] (1.9780798571642162,1.9780798571642115)-- (1.7192608120616968,2.9440056834532804);
\draw [line width=2pt] (1.7192608120616968,2.9440056834532804)-- (0.7533349857726285,3.2028247285558025);
\draw [line width=2pt] (0.7533349857726285,3.2028247285558025)-- (0,3.2028247285558034);
\draw [line width=2pt] (0,3.2028247285558034)-- (-0.9659258262890664,2.9440056834532835);
\draw [line width=2pt] (-0.9659258262890664,2.9440056834532835)-- (-1.224744871391588,1.9780798571642153);
\draw [line width=2pt] (-1.224744871391588,1.9780798571642153)-- (-1.2247448713915885,1.2247448713915894);
\draw [line width=2pt] (-1.2247448713915885,1.2247448713915894)-- (-0.9659258262890682,0.258819045102521);
\draw [line width=2pt] (-0.9659258262890682,0.258819045102521)-- (0,0);
\filldraw [color=black] (0,0) circle (2.5pt);
\filldraw [color=black] (0.1248368242936758,0.16913923218656002) node {$a_{11}$};
\filldraw [color=black] (0.753334985772626,0) circle (2.5pt);
\filldraw [color=black] (0.8801014679428514,0.16913923218656002) node {$a_{12}$};
\filldraw [color=black] (-0.9659258262890682,0.258819045102521) circle (2.5pt);
\filldraw [color=black] (-1.008010559856487,0.4260849356960728) node[below left] {$b_{1}$};
\filldraw [color=black] (-1.2247448713915885,1.2247448713915894) circle (2.5pt);
\filldraw [color=black] (-1.3247448713915885,1.2247448713915894) node[left] {$a_{42}$};
\filldraw [color=black] (-1.224744871391588,1.9780798571642153) circle (2.5pt);
\filldraw [color=black] (-1.324744871391588,1.9780798571642153) node[left] {$a_{41}$};
\filldraw [color=black] (-0.9659258262890664,2.9440056834532835) circle (2.5pt);
\filldraw [color=black] (-0.9518010559856487,3.112335472386434) node[left] {$b_{4}$};
\filldraw [color=black] (0,3.2028247285558034) circle (2.5pt);
\filldraw [color=black] (0.1248368242936758,3.3692811758959467) node {$a_{32}$};
\filldraw [color=black] (0.7533349857726285,3.2028247285558025) circle (2.5pt);
\filldraw [color=black] (0.8801014679428514,3.3692811758959467) node {$a_{31}$};
\filldraw [color=black] (1.7192608120616968,2.9440056834532804) circle (2.5pt);
\filldraw [color=black] (1.814449480704718,3.112335472386434) node {$b_{3}$};
\filldraw [color=black] (1.9780798571642162,1.9780798571642115) circle (2.5pt);
\filldraw [color=black] (2.0780798571642162,1.9780798571642115) node[right] {$a_{22}$};
\filldraw [color=black] (1.9780798571642149,1.2247448713915854) circle (2.5pt);
\filldraw [color=black] (2.0780798571642149,1.2247448713915854) node[right] {$a_{21}$};
\filldraw [color=black] (1.7192608120616921,0.2588190451025174) circle (2.5pt);
\filldraw [color=black] (1.814449480704718,0.4260849356960728) node[below right] {$b_{2}$};
\end{tikzpicture}
\caption{The dodecagon defined at the beginning of the construction}
\label{fig:12gon}
\end{figure}

\subsection{Construction of the Balanced Vertices in the Interior of the Dodecagon}

Denote by $p$ the center of the dodecagon. For example, $p$ is the intersection of $\overline{b_1b_3}$ and $\overline{b_2b_4}$.

For each $i\in \{1,2,3,4\}$, define the point $f_i$ as the Fermat point of the triangle $a_{i1}a_{12}p$, and connect $f_i$ with an edge to each of the three corners of this triangle. An example for $i=2$ can be seen in Figure \ref{fig:12gon_int+boundary}.

Recall that the Fermat point is the unique point $x$ in a triangle such that the angle at $x$ between any two corners of the triangle is $2\pi/3$. It exists as long as all interior angles of the triangle are less than $2\pi/3$. For this triangle, it is apparent that the angles are significantly smaller (less than $\pi/2$ in fact), so $f_i$ is indeed well-defined.

\subsection{Construction of the Four Boundary Vertices}

For each $i\in \{1,2,3,4\}$, we fix a boundary vertex $d_i$ through the following constraints. An example of the construction for $i=2$ is given in Figure \ref{fig:12gon_int+boundary}, denoted in red.
\begin{itemize}
    \item $a_{i1}d_ia_{i2}$ is an isosceles triangle.
    \item The interior angles at $a_{i1}$ and $a_{i2}$ are equal to $\beta$.
    \item The triangle lies outside the dodecagon.
\end{itemize}
Note that $\beta$ is provided by Lemma \ref{lemma:alpha_beta_existence}, which means that $\beta\in(0,\pi/2)$. Therefore, each $d_i$ is uniquely well-defined. 

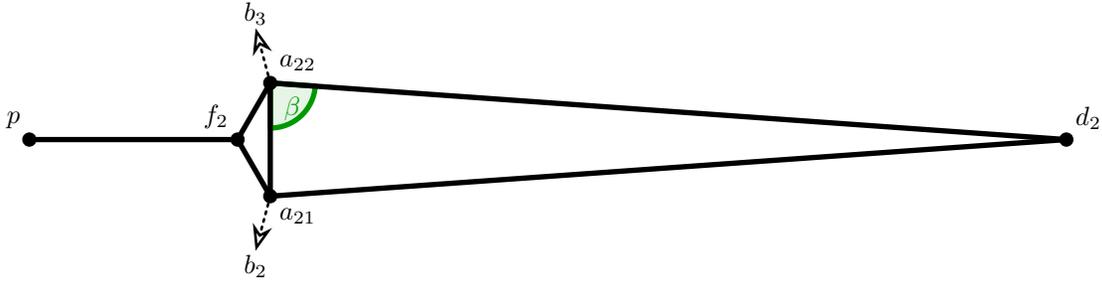
\begin{figure}[htbp]
\centering

\definecolor{qqzzqq}{rgb}{0,0.6,0}
\begin{tikzpicture}[line cap=round,line join=round,>=triangle 45,x=2cm,y=2cm]
\draw [line width=2pt,qqzzqq,fill=qqzzqq!10] (1.9780798571642162,1.9780798571642115) -- (1.9780798571642162, 1.6780798571642115) arc (-90:85.92832369838472-90:0.3) -- cycle;
\draw [line width=2pt] (1.9780798571642149,1.2247448713915854)-- (1.9780798571642162,1.9780798571642115);
\draw [line width=2pt] (7.269540714670677,1.6014123642778986)-- (1.9780798571642162,1.9780798571642115);
\draw [line width=2pt] (7.269540714670677,1.6014123642778986)-- (1.9780798571642149,1.2247448713915854);
\draw [line width=2pt] (1.9780798571642162,1.9780798571642115)-- (1.7606107787513212,1.6014123642778983);
\draw [line width=2pt] (1.7606107787513212,1.6014123642778983)-- (1.9780798571642149,1.2247448713915854);
\draw [line width=2pt] (1.7606107787513212,1.6014123642778983)-- (0.3766674928863143,1.6014123642779008);
\draw [line width=1pt,dash pattern=on 1pt off 2pt, \arrow{black}] (1.9780798571642162,1.9780798571642115)-- (1.8805911363342211,2.3419127164666547);
\draw [line width=1pt,dash pattern=on 1pt off 2pt, \arrow{black}] (1.9780798571642149,1.2247448713915854)-- (1.8805911363342183,0.860912012089142);
\draw [fill=black] (1.9780798571642162,1.9780798571642115) circle (2.5pt);
\draw[color=black] (1.9780798571642162,1.9980798571642115) node[above right] {$a_{22}$};
\draw [fill=black] (1.9780798571642149,1.2247448713915854) circle (2.5pt);
\draw[color=black] (1.9780798571642149,1.2047448713915854) node[below right] {$a_{21}$};
\draw [fill=black] (0.3766674928863143,1.6014123642779008) circle (2.5pt);
\draw[color=black] (0.3766674928863143,1.6114123642779008) node[above left] {$p$};
\draw [fill=black] (7.269540714670677,1.6014123642778986) circle (2.5pt);
\draw[color=black] (7.269540714670677,1.6114123642778986) node[above right] {$d_2$};
\draw [fill=black] (1.7606107787513212,1.6014123642778983) circle (2.5pt);
\draw[color=black] (1.7606107787513212,1.6114123642778983) node[above left] {$f_2$};
\draw[color=black] (1.8805911363342211,2.4419127164666547) node {$b_3$};
\draw[color=black] (1.8805911363342183,0.760912012089142) node {$b_2$};
\draw[qqzzqq] (2.0080798571642162,1.9480798571642115) node[below right] {$\beta$};
\end{tikzpicture}
\caption{Construction of interior vertices $p$ and $f_{2}$ as well as the boundary vertex $d_2$ with relevant connections to scale. The angle $\beta$ is labeled.}
\label{fig:12gon_int+boundary}
\end{figure}

\subsection{Construction of Additional Balanced Vertices}

For each $i\in \{1,2,3,4\}$ we define additional balanced vertices $c_i,e_i$ as follows.

Define $c_i$ as the point of intersection of $\overline{a_{(i-1)2}d_i}$ and and $\overline{a_{i1}d_{(i-1)}}$. We connect $c_i$ with an edge to all four of these vertices.

We will later use the following fact about angles:
\begin{lemma}\label{lemma:alpha}
    For each $i\in \{1,2,3,4\}$, consider the edge from $a_{i1}$ to $c_i$, and the edge from $a_{i1}$ to $a_{i2}$. The larger angle between these two edges at $a_{i1}$ is equal to $\alpha$ as given by Lemma \ref{lemma:alpha_beta_existence}.
\end{lemma}

The proof of this lemma is postponed to Section \ref{section:proof}.

Define $e_i$ as the Fermat point of the triangle $c_id_id_{(i-1)}$ and connect $e_i$ with an edge to the three corners of this triangle.

An example of the situation surrounding $b_i$, $c_i$ and $e_i$ is given in Figure \ref{fig:b_nbh} for $i=2$.

\begin{figure}[htbp]
\centering

\begin{tikzpicture}[line cap=round,line join=round,>=triangle 45,x=30cm,y=30cm, scale=0.8, rotate=90]
\draw [line width=2pt] (-0.9831319305486067,0.24161294084298168)-- (-0.9659258262890682,0.258819045102521);
\draw [line width=2pt] (-1.0799680129622857,0.14477685842930255)-- (-0.9831319305486067,0.24161294084298168);
\draw [line width=2pt,\arrow{black}](-0.9659258262890682,0.258819045102521)-- (-0.9895554703999815,0.34700607748922047);
\draw [line width=2pt,\arrow{black}](-0.9831319305486067,0.24161294084298168)-- (-1.0049208179551508,0.3302727239893786);
\draw [line width=2pt,\arrow{black}](-0.9831319305486067,0.24161294084298168)-- (-0.8944721474022097,0.21982405343643752);
\draw [line width=2pt,\arrow{black}](-0.9659258262890682,0.258819045102521)-- (-0.8777387939023688,0.23518940099160762);
\draw [line width=2pt,\arrow{black}](-0.9831319305486067,0.24161294084298168)-- (-0.9613430431420626,0.1529531576965847);
\draw [line width=2pt,\arrow{black}] (-1.0799680129622857,0.14477685842930255)-- (-1.0563383688513723,0.056589826042603125);
\draw [line width=2pt,\arrow{black}](-0.9831319305486067,0.24161294084298168)--(-1.0717917136950037,0.26340182824952585) ;
\draw [line width=2pt,\arrow{black}](-1.0799680129622857,0.14477685842930255)--(-1.168155045348985,0.16840650254021594) ;

\draw [fill=black] (-0.9659258262890682,0.258819045102521) circle (2.5pt);
\draw[color=black] (-0.9559258262890682,0.267819045102521) node {$b_{2}$};
\draw [fill=black] (-0.9831319305486067,0.24161294084298168) circle (2.5pt);
\draw[color=black] (-0.9699160319269784,0.22796892388889625) node {$c_{2}$};
\draw [fill=black] (-1.0799680129622857,0.14477685842930255) circle (2.5pt);
\draw[color=black] (-1.0709680129622857,0.14277685842930255) node {$e_{2}$};
\draw[color=black] (-1.173155045348985,0.17140650254021594) node {$d_1$};
\draw[color=black] (-1.0536383688513723,0.047589826042603125) node {$d_{2}$};
\draw[color=black] (-1.0777917136950037,0.26540182824952585) node {$d_1$};
\draw[color=black] (-0.9603430431420626,0.1439531576965847) node {$d_{2}$};
\draw[color=black] (-1.0099208179551508,0.3382727239893786) node {$a_{12}$};
\draw[color=black] (-0.9925554703999815,0.35500607748922047) node {$a_{12}$};
\draw[color=black] (-0.8874721474022097,0.21782405343643752) node {$a_{21}$};
\draw[color=black] (-0.8707387939023688,0.23318940099160762) node {$a_{21}$};

\end{tikzpicture}
\caption{Construction of vertices $c_{2}$ and $e_{2}$ with relevant connections to scale.}
\label{fig:b_nbh}
\end{figure}
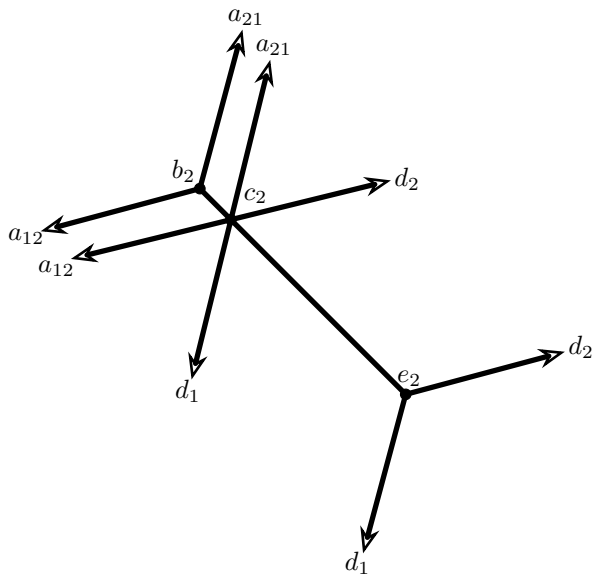

While the $c_i$ can be added without a problem, the existence of $e_i$ depends on the following lemma:

\begin{lemma}\label{lemma:e_i_fermatpoint}
    All three interior angles of the triangle $c_id_id_{(i-1)}$ are less than $2\pi/3$ (and therefore, their Fermat point $e_i$ exists).
\end{lemma}

The proof of this lemma is deferred to Section \ref{section:proof}.

\subsection{Proof of irreducible geodesic net properties}

For the constructed net to be an irreducible geodesic net with four boundary vertices, the following properties must be fulfilled:
\begin{itemize}
    \item None of the edges overlap, i.e. every edge has weight one.
    \item Except for the four boundary vertices, all vertices are balanced.
    \item The only subnets are trivial one-point subnets.
\end{itemize}

The first property is clear based on the construction of the net (the existence of edges with weight $2$ or more would require that some of the drawn segments are parallel, which is not the case). We will now prove the second and third properties.

\begin{lemma}
    Except for the boundary vertex set $S=\{d_1,d_2,d_3,d_4\}$, all vertices of the constructed geodesic net are balanced.
\end{lemma}

\begin{proof}

Most of the non-boundary vertices can easily be shown to be balanced:
\begin{itemize}
    \item Due to the symmetry of the construction, $p$ is a degree 4 balanced vertex where two straight lines intersect orthogonally.
    \item At $c_i$, six edges meet. Four of them form an intersection of two straight lines by construction. The other two lead to the Fermat points $b_i$ and $e_i$. Due to the symmetry of the net, they also form a straight line. Thus, $c_i$ is a degree 6 balanced vertex.
    \item Each $b_i$ is a Fermat point. More specifically, the dodecagon was designed such that the interior angle at $b_i$ is $2\pi/3$. Due to the reflective symmetry of the net, the other two angles at $b_i$ are equal to each other and therefore also $2\pi/3$.
    \item Each $e_i$ and $f_i$ is a Fermat point by virtue of the construction. It follows that each of them is a degree 3 balanced vertex.
\end{itemize}

Proving that the vertices $a_{ij}$ are balanced, on the other hand, is very much non-trivial. Making these vertices balanced was the reason for choosing $\alpha$ and $\beta$ as provided by \ref{lemma:alpha_beta_existence}.

Without loss of generality (due to rotational and reflective symmetry), we will argue that $a_{32}$ is balanced. A close-up of $a_{32}$ is shown in Figure \ref{fig:aij_exaggerated}, where the vertices with apostrophes denote the direction of the original vertices relative to $a_{32}$.

Notice that the direction (as a unit vector) of each edge emanating from $a_{32}$ can be written in the form of $(\cos \theta, \sin\theta)$, where $\theta\in [0, 2\pi)$ is the usual polar coordinate angle. As in all figures so far, we orient the net so that the edge from $a_{32}$ to $a_{31}$ is horizontal. If we denote by $\theta_q$ the respective angle for the edge from $a_{32}$ to the vertex $q$, we can write:
\begin{itemize}
    \item $\theta_{a_{31}}=0$, by choice,
    \item $\theta_{d_{3}}=\beta$, since $d_3$ was defined specifically through this angle,
    \item $\theta_{c_{4}}=\alpha$, by Lemma \ref{lemma:alpha},
    \item $\theta_{b_{4}}=2\pi-\frac{11\pi}{12}=\frac{13\pi}{12}$, since $\frac{11\pi}{12}$ is the interior angle of the dodecagon here,
    \item $\theta_{f_{3}}=2\pi-\frac{\pi}{6}=\frac{11\pi}{6}$, since $\frac{\pi}{6}$ is the interior angle of the isosceles triangle $a_{32}f_3a_{31}$ whose angle at $f_3$ is $2\pi/3$ ($f_3$ is a Fermat point).
\end{itemize}

Adding these unit vectors together gives the following conditions on $a_{32}$ to be balanced:
\begin{align*}
    1 + \cos \beta + \cos\alpha + \cos\frac{13\pi}{12} + \cos\frac{11\pi}{6} =&\; 0,\\
    0 + \sin \beta + \sin\alpha + \sin\frac{13\pi}{12} + \sin\frac{11\pi}{6} =&\; 0,
\end{align*}
These are exactly the equations that were chosen to define the angles $\alpha$ and $\beta$. It follows that $a_{32}$ is balanced.

\definecolor{qqwwzz}{rgb}{0,0.4,0.6}
\definecolor{qqzzqq}{rgb}{0,0.6,0}

\begin{figure}[t]
\begin{tikzpicture}[line cap=round,line join=round,>=triangle 45,x=50cm,y=50cm, scale=0.8]
\draw [shift={(0,3.2028247285558034)},line width=1pt,color=qqzzqq,fill=qqzzqq,fill opacity=0.1] (0,0) -- (0:0.014793874551972647) arc (0:85.92832369838472:0.014793874551972647) -- cycle;
\draw [shift={(0,3.2028247285558034)},line width=2pt,color=qqwwzz,fill=qqwwzz,fill opacity=0.1] (0,0) -- (0:0.02219081182795897) arc (0:187.29401946379016:0.02219081182795897) -- cycle;
\draw [line width=1pt, \arrow{black}] (0,3.2028247285558034)-- (-0.09057055540886011,3.1912319815915904);
\draw [line width=1pt, \arrow{black}] (0,3.2028247285558034)-- (-0.08818703238669723,3.17919508444489);
\draw [line width=1pt, \arrow{black}] (0,3.2028247285558034)--(0.07906633020120664,3.157175761530302);
\draw [line width=1pt, \arrow{black}] (0,3.2028247285558034)--(0.09129793405100468,3.2028247285558034);
\draw [line width=1pt, \arrow{black}] (0,3.2028247285558034)--(0.0064825512808915605,3.2938922270987625);
\draw [shift={(0,3.2028247285558034)},line width=1pt,color=qqwwzz] (0:0.02219081182795897) arc (0:187.29401946379016:0.02219081182795897);
\draw [shift={(0,3.2028247285558034)},line width=1pt,color=qqwwzz] (0:0.019786807213263414) arc (0:187.29401946379016:0.019786807213263414);

\draw [fill=black] (0,3.2028247285558034) circle (2.5pt);
\draw[color=black] (-0.003,3.1965947611166805) node {$a32$};
\draw[color=black] (0.0084825512808915605,3.2938922270987625) node[above] {$d_3$};
\draw[color=black] (-0.08618703238669723,3.17719508444489) node[left] {$b_4$};
\draw[color=black] (-0.09057055540886011,3.1907319815915904) node[left] {$c_4$};
\draw[color=black] (0.09129793405100468,3.2028247285558034) node[right] {$a_{31}$};
\draw[color=black] (0.07706633020120664,3.156175761530302) node[right] {$f_3$};
\draw[color=qqzzqq] (0.0110167195911491,3.2093886356686536) node[left] {$\beta$};
\draw[color=qqwwzz] (-0.003177154960823548,3.2132378699876503) node[left] {$\alpha$};

\end{tikzpicture}
\caption{Close-up structure of vertex $a_{32}$ with angles exaggerated.}
\label{fig:aij_exaggerated}
\end{figure}
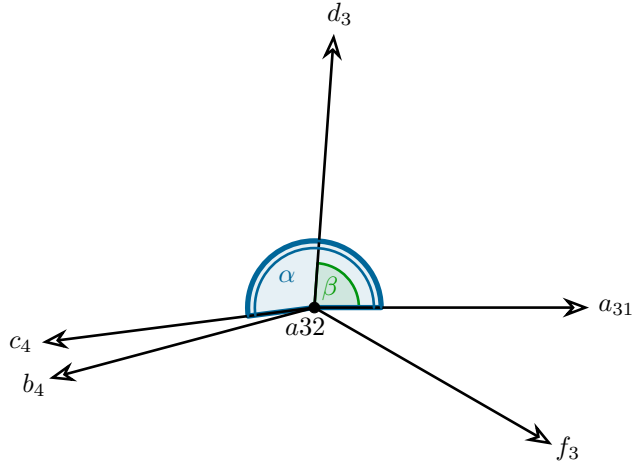
\end{proof}

We will now consider the last property to prove.

\begin{lemma}
    The geodesic net as constructed is irreducible, i.e. it has no non-trivial subnets.
\end{lemma}

\begin{proof}
    Denote the net as constructed by $G$ and let $G'$ be a subnet of $G$. As a first case, assume that at least one edge incident to $a_{32}$ (shown in figure \ref{fig:aij_exaggerated}) is not contained in $G'$. 
    
    Generally, this would leave the following options for a  $G'$ and its relation to $a_{32}$:
    \begin{enumerate}
        \item $G'$ does not contain $a_{32}$.
        \item $G'$ contains $a_{32}$ and \textit{exactly one} edge incident to it. This is impossible since degree 1 vertices can't be balanced.
        \item $G'$ contains $a_{32}$ and \textit{exactly two} edges incident to it. This is impossible since it would require the two edges to form a straight line but no such subset of edges exist.
        \item $G'$ contains $a_{32}$ and \textit{exactly three} edges incident to it. This is impossible since it would require three edges with mutual angles being $2\pi/3$ but no such subset of edges exists (see the previous proof for the exact angles).
        \item $G'$ contains $a_{32}$ and \textit{exactly four} edges incident to it. This is impossible since it would require the four edges to form two straight lines but no such subset of edges exist.
    \end{enumerate}
    
    Only the first option is therefore possible. We will argue that it implies that $G'$ is trivial.

    For that purpose, assume that $G'$ does not contain $a_{32}$ (and therefore no edges incident to it). This means that the neighbor $a_{31}$ would lose one of its edges. By an argument analogous to the one above, $a_{31}$ must not be in $G'$ at all. This means that its neighbor $b_3$ (a degree 3 Fermat point) would lose one of its edges and its remaining two edges can't balance $b_3$. Therefore, $b_3$ can't be in $G'$. We can now continue analogously around the entire dodecagon to conclude that none of the $b_i$ and $a_{ij}$ are in $G'$.
    
    $p$ as well as the four $f_i$ are now disconnected from the boundary vertices and can't be in $G'$.

    At this point, each $c_i$ lost three of its edges, all on one side (compare figure \ref{fig:b_nbh}) which means none of the $c_i$ can remain in $G'$ by just using the remaining three edges. Lastly, each Fermat point $e_i$ now lost one of its edges and therefore can't be in $G'$. At this point, only the $d_i$ are left. Since $G'$ must be connected, it must now be a trivial geodesic net consisting of a single $d_i$.

    The same argument can be repeated under the assumption that any other $a_{ij}$ is missing at least one edge. This means that for $G'$ to be nontrivial, it must contain all of the $a_{ij}$ and all edges incident to those vertices. By ``inverting'' the above argument, we could now argue that the entire dodecagon must be in $G'$, that then also all other vertices must be in $G'$, and finally arrive at the point that $G'=G$.

    This proves that $G$ is indeed irreducible.
\end{proof}

\section{Towards a series of nets}
\label{section:outro}

In the context of Conjecture \ref{conj:fourandinfinite}, we consider the present construction quite promising as it makes use of several properties and methods some of which were (to our knowledge) not previously employed. They could help with the construction of a series of geodesic nets as needed to  prove the Conjecture. In the following, we would like to highlight a few of them.

\subsection{Using ``irregular'' vertices}

It is an inherent property of degree 3 and degree 4 balanced vertices that they are highly regular in their geometry: Degree 3 vertices must look like a Fermat point with equal angles between edges, whereas degree 4 vertices are always just the intersection of two straight lines. In other words, only at degree 5 and above can the incident edges have an irregular distribution, where ``irregular'' means that the incident edges do not exhibit any rotational or reflective symmetries.

This leads us to an important qualitative distinction between the example previously constructed in \cite{Parsch2019-fz} and the present one. Namely, while the former does make use of degree 5 vertices, these vertices do have a reflective symmetry. The present example, however, includes irregular vertices as shown in figure \ref{fig:aij_exaggerated} that have no symmetries (neither as a whole, nor for subsets of three or four edges).

We consider this a quite promising development, since such irregular vertices are a lot more versatile as a tool to balance a net. In our case, we used them as the main ``wiggle room'' when finding a candidate net through gradient descent (see below).

\subsection{``Canceling'' imbalance at the centre}

There is another geometric approach that wasn't employed in \cite{Parsch2019-fz}. The present net has a vertex in the centre at which some of the imbalance that is produced by vertices arranged around the centre is being ``canceled'' by joining edges there. This is quite useful, and it seems highly likely that any series in the sense of the Conjecture would need to make use of this approach. We think this is the case since the four boundary vertices -- due to their position relative to the ``ring'' of vertices -- can only handle a limited amount of imbalance per vertex.

It seems to be a promising approach for future constructions to employ this ``cancelation at the centre'' several (likely: an ever increasing number of) times to allow for more vertices in the ''ring'' region.

\subsection{Using gradient descent to find candidate nets} While we provide a proof that the net as constructed is balanced, we originally discovered its shape through an algorithmic approach.

Based on the shape of the net constructed in \cite{Parsch2019-fz}, we developed a dodecagon (instead of octagon) based shape at the centre, adding additional Fermat-point-like vertices to achieve a \textit{roughly} balanced net based on visual inspection.

Specifically, we fixed a dodecagon of $b_i$ and $a_{ij}$ with the same angles as as in figure \ref{fig:12gon}, but (for now) with all equal sides. We then added the $f_i$ as Fermat points towards the centre. All these vertices were kept fixed. We then introduced the $c_i$ as points that lie on the two lines through $b_1/b_3$ and $b_2/b_4$, but kept their distance from the centre $p$ as a parameter $t=d(c_i,p)$. Once $t$ is chosen and the $c_i$ are fixed, each $d_i$ was also fixed. For example, $d_2$ lies on the intersection of the extended line segments $\overline{a_{12}c_2}$ and $\overline{a_{31}c_3}$. Lastly, the $e_i$ as Fermat points were fixed as well.

This means that for any given value of $t>d(b_i, p)$ that makes the $d_i$ well-defined, we get a unique net, always of the same topology. All $b_i$, $e_i$ and $f_i$ are balanced as Fermat points, all $c_i$ are balanced as degree six balanced vertices. Only the $a_{ij}$ were not necessarily balanced, leaving us necessary ``wiggle room''.

Using computer software, we devised a tool that computed an approximation of the imbalance $\mathrm{imb}(a_{ij})$ for any given value of $t$. For a vertex $v$, $\mathrm{imb}(v)$ refers to the norm of the sum of all edge-unit-vectors at $v$, i.e. $v$ is balanced if and only if $\mathrm{imb}(v)=0$ (see also \cite{Nabutovsky2019-iy}). After some experimentation, we observed that, indeed, for a small range of $t$-values, the approximation of $\mathrm{imb}(a_{ij})$ appeared to be very close to zero.

We then moved on to further empirical verification using gradient descent. Consider the total imbalance of all non-boundary vertices, i.e.:
\begin{align*}
    L=\sum_{v\not\in \{d_1,d_2,d_3,d_4\}}\mathrm{imb}(v)
\end{align*}
If we fix all but one vertex $v_0$ and consider $L(v_0)$ as a function of the position of this single vertex, one can see that $L$ is smooth, and that $L\geq 0$ with equality if and only if $G$ is a geodesic net where only the four $d_i$ are unbalanced vertices.

This is why we used a step-by-step gradient descent with $L$ as a loss function as follows: We first fixed the positions of all vertices close to those of the seemingly-balanced net from above. Using an algorithm in python, we then went through all non-boundary vertices cyclically, each time adjusting the position of the vertex based on its sum of incident unit vectors. This algorithm produced very strong empirical evidence that the given topology can in fact produce a balanced net. Namely, for a wide variety of step sizes and even for completely ``off-the-chart'' initial positions for the vertices, the net seemed to stabilize in a balanced configuration.

Equipped with this empirical backing, we then embarked on a proof based on the idea that matching angles $\alpha$ and $\beta$ as required by Lemma \ref{lemma:alpha_beta_existence} can be chosen to get a truly balanced net.

The gradient descent algorithm and an animation showing the process for one example can be found in the Github repository \href{https://github.com/zhr98971/GeoNet}{https://github.com/zhr98971/GeoNet}, under the files named \texttt{algorithm\_4\_25.ipynb} and \texttt{animation\_4\_25.gif} respectively.

\subsection{Attempts at generalizing the construction} The two geodesic nets in Figure \ref{fig:twoirreduciblenets} appear to be quite similar. Specifically, note that they have an octagon and dodecagon ``ring'' of vertices respectively. This leads to a natural follow-up question.

\begin{question}
Does there exist a series $G_n, n\geq 2$ of geodesic nets similar in shape to those in Figure \ref{fig:twoirreduciblenets}, so that each $G_n$ has a ``ring'' of $4n$ vertices (i.e. so that the two nets in the figure are $G_2$ and $G_3$ of such a series)?
\end{question}

So far we considered several possible structures of such geodesic nets, especially those with $8n+4$ vertices in the ``ring", and used the same optimization techniques involving the loss function $L$ and gradient descent. While these attempts so far only resulted in the algorithm converging towards what appear to be multinets (i.e. geodesic nets with integer-weighted edges), we are cautiously optimistic that an ``inductive construction'' in the sense of the above question can be found. This would provide a series as desired by Conjecture \ref{conj:fourandinfinite}.

\section{Auxiliary Proofs}\label{section:proof}

There are three rather technical proofs that we previously deferred to this section.

\begin{proof}(of Lemma \ref{lemma:alpha_beta_existence})
    Rearrange the equations as follows:
    \begin{align*}
        \cos \beta =&\; -1 - \cos\alpha - \cos\frac{13\pi}{12} -\cos\frac{11\pi}{6},\\
        \sin \beta =&\; - \sin\alpha - \sin\frac{13\pi}{12} - \sin\frac{11\pi}{6}.
    \end{align*}
    We have the given restriction that $\beta\in (0, \frac{\pi}{2})$, which implies that the above equations are true if and only if the following equations are true:
    \begin{align*}
        \beta =&\; f(\alpha):=\arccos\left(-1 - \cos\alpha - \cos\frac{13\pi}{12} -\cos\frac{11\pi}{6}\right),\\
        \beta =&\; g(\alpha):=\arcsin\left(- \sin\alpha - \sin\frac{13\pi}{12} - \sin\frac{11\pi}{6}\right).
    \end{align*}
    The functions $f$ and $g$ that we just defined are only well-defined if these inequalities are true:
    \begin{align*}
        -1\leq -1 - \cos\alpha - \cos\frac{13\pi}{12} - \cos\frac{11\pi}{6}\leq 1\\
        -1\leq - \sin\alpha - \sin\frac{13\pi}{12} - \sin\frac{11\pi}{6}\stackrel{(*)}{\leq} 1
    \end{align*}
    Due to the restriction that $\alpha\in(\pi,13\pi/12)$, the first expression is decreasing in $\alpha$. Verifying the values at $\alpha=\pi$ and $\alpha=13\pi/12$, it follows that the bounds are fulfilled.
    
    On the other hand, the second expression is increasing in $\alpha$. Its minimum at $\alpha=\pi$ is within the lower bound. The upper bound (marked with $(*)$) is only fulfilled up to $\alpha=K:=\pi-\arcsin\left(-1- \sin\frac{13\pi}{12} - \sin\frac{11\pi}{6}\right)<13\pi/12$.

    Therefore, there is no solution $(\alpha,\beta)$ for the system of equations with $\alpha\in (K,13\pi/2)$. Consequently, we can focus on $\alpha\in(\pi,K]$ to see if there is a unique solution for that range.
    
    \textbf{Claim 1:} The functions $f$ and $g$ are well defined, continuous on $[\pi,K]$ and $C^1$ on $\alpha\in(\pi,K)$.

    \textit{Proof of Claim 1:} We established that $f$ and $g$ are well-defined up to $\alpha=K$ above. It is now immediate that they are continous and $C^1$ as claimed, since both of them are defined through a combination of functions having those properties.

    \textbf{Claim 2:} There is a unique value $\alpha_0\in(\pi,K)$ such that $f(\alpha_0)=g(\alpha_0)$.

    \textit{Proof of Claim 2:} For the function $h(\alpha):=f(\alpha)-g(\alpha)$, note that $h(\pi)\approx 0.6092$ whereas $h(K)\approx -0.0704$. By the Intermediate Value Theorem (which applies thanks to Claim 1), an $\alpha_0\in(\pi,K)$ such that $h(\alpha_0)=0$ exists.

    For the uniqueness, it is a rather tedious but straightforward process to verify that for $\alpha\in(0,K)$:
    \begin{align*}
        h'(\alpha)=\frac{\sin\alpha}{\sqrt{1-\left(-1-\cos\alpha - \cos\frac{13\pi}{2} - \cos\frac{11\pi}{6}\right)^2}}-\frac{\cos\alpha}{\sqrt{1-\left(- \sin\alpha - \sin\frac{13\pi}{2} -\sin\frac{11\pi}{6}\right)^2}}< 0
    \end{align*}
    So $h$ is injective in the given range and there is indeed only one such $\alpha_0$.

    Now that Claim 2 is proven, recall that $(\alpha,\beta)$ is a solution to the original system of equations if and only if $\beta=f(\alpha)=g(\alpha)$. Therefore $(\alpha_0,f(\alpha_0))=(\alpha_0,g(\alpha_0))$ is the unique solution as required by the lemma.
\end{proof}

\begin{remark}
    \label{rmk:aij_balanced}
    The exact solution $(\alpha,\beta)$ yielded by this lemma is in fact known to us. It was found using a Python script involving the package \texttt{sympy}. However, the expressions of the two angles are rather complicated and tedious, and establishing that they do in fact solve the system of equations requires significant brute-force work spreading across several pages. This is why we are only proving the existence of the solution here. 
\end{remark}

\begin{proof}(of Lemma \ref{lemma:alpha}) Due to symmetry, we will focus on the case $i=2$ which is shown in Figure \ref{fig:angles_exaggerated}. Note that some angles are slightly exaggerated for better visibility.

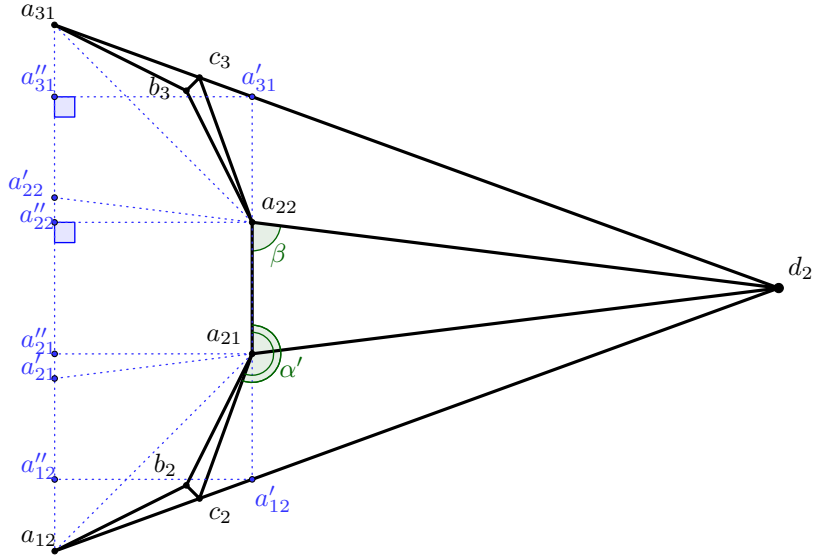
\begin{figure}[htbp]
\centering
\definecolor{qqqqff}{rgb}{0,0,1}
\definecolor{qqwuqq}{rgb}{0,0.39215686274509803,0}
\definecolor{ttttff}{rgb}{0.2,0.2,1}

\begin{tikzpicture}[line cap=round,line join=round,>=triangle 45,x=2.5cm,y=2.5cm, scale=0.7]
\draw [shift={(0.5,2)},line width=0.5pt,color=qqwuqq,fill=qqwuqq,fill opacity=0.10000000149011612] (0,0) -- (-90:0.21748219413160372) arc (-90:-7.1250163489017995:0.21748219413160372) -- cycle;
\draw [shift={(0.5,1)},line width=0.5pt,color=qqwuqq,fill=qqwuqq,fill opacity=0.10000000149011612] (0,0) -- (-109.98310652189998:0.21748219413160372) arc (-109.98310652189998:90:0.21748219413160372) -- cycle;
\draw[line width=0.5pt,color=qqqqff,fill=qqqqff,fill opacity=0.1] (-1.0015452569526568,2.800059202936838) -- (-0.8477621395540563,2.8001312120787194) -- (-0.8478341486959378,2.95391432947732) -- (-1.0016172660945384,2.9538423203354385) -- cycle; 
\draw[line width=0.5pt,color=qqqqff,fill=qqqqff,fill opacity=0.10000000149011612] (-1.0010982899177807,1.8455139576842707) -- (-0.8473151725191802,1.8455859668261523) -- (-0.8473871816610616,1.999369084224753) -- (-1.0011702990596623,1.9992970750828714) -- cycle; 
\draw [line width=1.2pt] (0,3)-- (0.5,2);
\draw [line width=1.2pt] (0.5,2)-- (0.5,1);
\draw [line width=1.2pt] (0.5,1)-- (0,0);
\draw [line width=1.2pt] (0.5,2)-- (4.5,1.5);
\draw [line width=1.2pt] (4.5,1.5)-- (0.5,1);
\draw [line width=1.2pt] (0,3)-- (0.1,3.1);
\draw [line width=1.2pt] (4.5,1.5)-- (0.1,3.1);
\draw [line width=1.2pt] (0.1,3.1)-- (0.5,2);
\draw [line width=1.2pt] (0.1,-0.1)-- (4.5,1.5);
\draw [line width=1.2pt] (0,0)-- (0.1,-0.1);
\draw [line width=1.2pt] (0.1,-0.1)-- (0.5,1);
\draw [line width=1.2pt] (-1.0018733241051274,3.5006812087655006)-- (0,3);
\draw [line width=1.2pt] (-1.0018733241051274,3.5006812087655006)-- (0.1,3.1);
\draw [line width=1.2pt] (-1,-0.5)-- (0.1,-0.1);
\draw [line width=1.2pt] (0,0)-- (-1,-0.5);
\draw [line width=0.5pt,dash pattern=on 0.5pt off 2pt,color=ttttff] (-1.0018733241051274,3.5006812087655006)-- (-1,-0.5);
\draw [line width=0.5pt,dash pattern=on 0.5pt off 2pt,color=ttttff] (-1.0012584989822628,2.187657312372783)-- (0.5,2);
\draw [line width=0.5pt,dash pattern=on 0.5pt off 2pt,color=ttttff] (-1.0006145438375893,0.8124231820203013)-- (0.5,1);
\draw [line width=0.5pt,dash pattern=on 0.5pt off 2pt,color=ttttff] (0.5,2.954545454545454)-- (0.5,0.04545454545454544);
\draw [line width=0.5pt,dash pattern=on 0.5pt off 2pt,color=ttttff] (0.5,2)-- (-1.0011702990596623,1.9992970750828714);
\draw [line width=0.5pt,dash pattern=on 0.5pt off 2pt,color=ttttff] (0.5,1)-- (-1.0007020478802684,0.9992972943420864);
\draw [shift={(0.5,1)},line width=0.5pt,color=qqwuqq] (-109.98310652189998:0.21748219413160372) arc (-109.98310652189998:90:0.21748219413160372);
\draw [shift={(0.5,1)},line width=0.5pt,color=qqwuqq] (-109.98310652189998:0.16311164559870278) arc (-109.98310652189998:90:0.16311164559870278);
\draw [line width=0.5pt,dash pattern=on 0.5pt off 2pt,color=ttttff] (-1.0018733241051274,3.5006812087655006)-- (0.5,2);
\draw [line width=0.5pt,dash pattern=on 0.5pt off 2pt,color=ttttff] (0.5,1)-- (-1,-0.5);
\draw [line width=0.5pt,dash pattern=on 0.5pt off 2pt,color=ttttff] (-1.0016172660945384,2.9538423203354385)-- (0.5,2.954545454545454);

\draw [fill=black] (0,0) circle (1.5pt);
\draw[color=black] (0,0) node[above left] {$b_{2}$};
\draw [fill=black] (4.5,1.5) circle (2.5pt);
\draw[color=black] (4.5,1.5) node[above right] {$d_{2}$};
\draw [fill=black] (0,3) circle (1.5pt);
\draw[color=black] (-0.04,3.025) node[left] {$b_{3}$};
\draw [fill=black] (0.5,2) circle (1.5pt);
\draw[color=black] (0.5,2) node[above right] {$a_{22}$};
\draw [fill=black] (0.5,1) circle (1.5pt);
\draw[color=black] (0.5,1) node[above left] {$a_{21}$};
\draw [fill=black] (0.1,3.1) circle (1.5pt);
\draw[color=black] (0.1,3.1) node[above right] {$c_{3}$};
\draw [fill=black] (0.1,-0.1) circle (1.5pt);
\draw[color=black] (0.1,-0.1) node[below right] {$c_{2}$};
\draw [fill=black] (-1.0018733241051274,3.5006812087655006) circle (1.5pt);
\draw[color=black] (-1.12,3.6) node {$a_{31}$};
\draw [fill=black] (-1,-0.5) circle (1.5pt);
\draw[color=black] (-1.12,-0.4) node {$a_{12}$};
\draw [fill=ttttff] (-1.0012584989822628,2.187657312372783) circle (1.5pt);
\draw[color=ttttff] (-1.0052286612538641,2.3) node[left] {$a_{22}'$};
\draw [fill=ttttff] (-1.0006145438375893,0.8124231820203013) circle (1.5pt);
\draw[color=ttttff] (-1.12,0.9) node {$a_{21}'$};
\draw [fill=ttttff] (0.5,2.954545454545454) circle (1.5pt);
\draw[color=ttttff] (0.5571466976673616,3.101961551587272) node {$a_{31}'$};
\draw[line width=0.5pt,dash pattern=on 0.5pt off 2pt,color=ttttff] (0.5,0.04545454545454544)--(-1.000255080845392,0.04545454545454544);
\draw [fill=ttttff] (0.5,0.04545454545454544) circle (1.5pt);
\draw[color=ttttff] (0.45,0.06545454545454544) node[below right] {$a_{12}'$};
\draw [fill=ttttff] (-1.0011702990596623,1.9992970750828714) circle (1.5pt);
\draw[color=ttttff] (-1.13,2.05) node {$a_{22}''$};
\draw [fill=ttttff] (-1.0007020478802684,0.9992972943420864) circle (1.5pt);
\draw[color=ttttff] (-1.12,1.1) node {$a_{21}''$};
\draw[color=qqwuqq] (0.7,1.75) node{$\beta$};
\draw[color=qqwuqq] (0.8,0.9) node {$\alpha'$};
\draw [fill=ttttff] (-1.0016172660945384,2.9538423203354385) circle (1.5pt);
\draw[color=ttttff] (-1.12,3.1) node{$a_{31}''$};
\draw [fill=ttttff] (-1.000255080845392,0.04475204908951896) circle (1.5pt);
\draw[color=ttttff] (-1.12,0.15) node {$a_{12}''$};

\end{tikzpicture}
\caption{Exaggerated local structure of the geodesic net in Figure \ref{fig:12gon}.}
\label{fig:angles_exaggerated}
\end{figure}

We define the following additional points as seen in the figure:

\begin{itemize}
    \item $a_{22}'$ is the intersection of the line through $d_2$ and $a_{22}$ with $\overline{a_{31}a_{12}}$.
    \item $a_{21}'$ is the intersection of the line through $d_2$ and $a_{21}$ with $\overline{a_{31}a_{12}}$.
    \item $a_{31}'$ is the intersection of the line through $a_{22}$ and $a_{21}$ with $\overline{a_{31}d_2}$.
    \item $a_{12}'$ is the intersection of the line through $a_{22}$ and $a_{21}$ with $\overline{a_{12}d_2}$.
    \item $a_{12}''$, $a_{21}''$, $a_{22}''$, $a_{31}''$ are the orthogonal projections onto $\overline{a_{31}a_{12}}$ of $a_{12}'$, $a_{21}$, $a_{22}$, $a_{31}'$ respectively.
\end{itemize}

Define the angle $\alpha':=\angle c_2a_{21}a_{22}$. To prove the lemma, we need to establish that $\alpha'=\alpha$ (recall that $\alpha$ is provided by Lemma \ref{lemma:alpha_beta_existence} and -- like $\beta$ -- was used in the construction of the net).

We will do so through the following equations:
\begin{align}
    \frac{d(a_{31}',a_{22})}{d(a_{31}, a_{22}')} &=\;\frac{d(a_{21},a_{22})}{d(a_{21}',a_{22}')}\\
    d(a_{21},a_{22})&=\;\frac{\sqrt{6}\cdot (1 - \tan(\alpha))}{\tan(\alpha)\cdot \tan(\beta) - 1}\\
    d(a_{21}',a_{22}')&=\;\frac{\sqrt{6}\cdot (1 - \tan(\alpha))}{\tan(\alpha)\cdot \tan(\beta) - 1}+\sqrt{6}\cdot \cot(\beta)\\
    d(a_{31}, a_{22}')&=\;\frac{\sqrt{6}}{2}\cdot (1-\cot(\beta))\\
    d(a_{31}',a_{22})&=\;\frac{\sqrt{6}}{2}\cdot \left(1-\cot\left(\frac{3}{2}\pi - \alpha'\right)\right)
\end{align}
Before proving each of these equations, we will show that they imply that $\alpha'=\alpha$ as desired.

Plugging the other four equations into the first one yields
\begin{align}
    \label{equ:scale_reduced}
    \frac{\frac{\sqrt{6}}{2}\cdot (1-\cot(\frac{3}{2}\pi - \alpha'))}{\frac{\sqrt{6}}{2}\cdot (1-\cot(\beta))} =&\; \frac{\frac{\sqrt{6}\cdot (1 - \tan(\alpha))}{\tan(\alpha)\cdot \tan(\beta) - 1}}{\frac{\sqrt{6}\cdot (1 - \tan(\alpha))}{\tan(\alpha)\cdot \tan(\beta) - 1}+\sqrt{6}\cdot \cot(\beta)}.
\end{align}

After some elementary simplification, equation \ref{equ:scale_reduced} becomes:
\begin{align}
    \label{equ:scale_simplified}
    1-\cot\left(\frac{3}{2}\pi - \alpha'\right)=&\;1 - \tan(\alpha).
\end{align}

Using the cotangent angle difference identity given by  $\cot(\theta-\varphi)=\frac{\cot(\theta)\cot(\varphi)+1}{\cot(\varphi)-\cot(\theta)}$, we have: 
\begin{align}
    \cot\left(\frac{3}{2}\pi - \alpha'\right) =\frac{\cot\left(\frac{3}{2}\pi\right)\cot(\alpha')+1}{\cot(\alpha')-\cot\left(\frac{3}{2}\pi\right)}=\frac{0\cdot \cot(\alpha')+1}{\cot(\alpha')-0}=\tan(\alpha').
\end{align}
We can use this to simplify the LHS of equation \ref{equ:scale_simplified} to get:
\begin{align}
    1-\tan(\alpha')=&\;1 - \tan(\alpha).
    \label{equ:scale_final}
\end{align}
and therefore $\tan\alpha'=\tan\alpha$. Generally, due to the periodicity of the tangent, the two angle arguments could differ. However, $\alpha\in(\pi, \frac{13\pi}{12})\subset (\pi,3\pi/2)$ by choice in Lemma \ref{lemma:alpha_beta_existence}, and it is apparent from the construction that $\alpha'$ is a reflex angle also in the range $(\pi,3\pi/2)$. Since the tangent is injective in this range, we get $\alpha=\alpha'$. This is the desired result.

We are left to argue equations (1) to (5). Equation (1) follows directly from the Intercept Theorem and the fact that $\overline{a_{31}a_{12}}$ and $\overline{a_{31}'a_{12}'}$ are two parallel line segments intersecting rays emanating from $d_2$. Equation (2) is true since we \textit{defined} this to be the distance during the construction of the dodecagon. Equations (3), (4) and (5) warrant individual proofs.

\textbf{Proof of equations (3) and (4):} Note that $\angle c_2a_{21}a_{22}=2\pi-\alpha'$. Notice that the net admits a reflective symmetry on $(d_2,d_4)$ and $(d_1,d_3)$, so clearly $\angle c_2a_{21}a_{22}=\angle c_3a_{22}a_{21}=2\pi - \alpha'$.

Let $a_{22}'$ be the intersect of the extension of $(d_2, a_{22})$ onto $(a_{12},a_{31})$, and let $a_{22}''$ be the unique point on $(a_{12},a_{31})$ such that $(a_{22}, a_{22}'')\perp (a_{12},a_{31})$; define $a_{21}'$ and $a_{21}''$ similarly (see Figure \ref{fig:angles_exaggerated}). Since $(a_{12}, a_{31})\parallel(a_{22}, a_{21})$, we have $\angle a_{22}''a_{22}'a_{22}=\angle a_{21}a_{22}d_2=\beta$ by definition.

Recall that by construction, $d(a_{31}, b_3)=d(b_3,a_{22})=1$, and $\angle a_{31}b_3a_{22}=\frac{2}{3}\pi$ radians. It then naturally follows that $d(a_{31},a_{22})=1\cdot \sqrt{3}=\sqrt{3}$ (well known triangle ratio). 

Further notice that the net admits a reflective symmetry on $(b_1, b_3)$. This means that the points $a_{31},a_{22}'',a_{22}$ form a right angle isosceles triangle with $\angle a_{31}a_{22}''a_{22}=\frac{\pi}{2}$. This means that $d(a_{31},a_{22}'')=d(a_{22}'',a_{22})=\frac{\sqrt{3}}{\sqrt{2}}$ (well known triangle ratio).

For simplicity, denote $M:=d(a_{22}'',a_{22}')$. Then using the above found values, we must have $\tan(\beta)=\frac{\sqrt{3}/\sqrt{2}}{M}$, meaning $M=\frac{\sqrt{6}}{2}\cdot \cot(\beta)$. Since $a_{21}''$ and $a_{22}''$ are perpendicular extensions from $a_{21}$ and $a_{22}$ respectively, we must have $d(a_{21},a_{22})=L$ by definition. Thus, $d(a_{21}',a_{22}')=d(a_{21}',a_{21}'')+d(a_{21}'',a_{22}'')+d(a_{22}'',a_{22}')=2\cdot M+L=L+\sqrt{6}\cdot \cot(\beta)=\frac{\sqrt{6}\cdot (1 - \tan(\alpha))}{\tan(\alpha)\cdot \tan(\beta) - 1}+\sqrt{6}\cdot \cot(\beta)$, and $d(a_{31}, a_{22}')=\frac{\sqrt{3}}{\sqrt{2}}-M=\frac{\sqrt{6}}{2}\cdot (1-\cot(\beta))$. These are equations (3) and (4) respectively.

\textbf{Proof of equation (5):} Let $a_{31}'$ be the intersect of the extension of $(a_{21}, a_{22})$ onto $(d_2,a_{31})$, and let $a_{31}''$ be the unique point on $(a_{12},a_{31})$ such that $(a_{31}', a_{31}'')\perp (a_{12},a_{31})$; define $a_{12}'$ and $a_{12}''$ similarly (see Figure \ref{fig:angles_exaggerated}). 

Since $\angle c_3a_{22}a_{21}=2\pi - \alpha'$, we must have $\angle c_3a_{22}a_{22}''=\frac{3}{2}\pi - \alpha'$. By reflective symmetry on $(b_1, b_3)$, we also have $\angle a_{31}'a_{31}a_{31}''=\angle c_3a_{31}a_{22}''=\angle c_3a_{22}a_{22}''=\frac{3}{2}\pi - \alpha'$. Since the vertices $a_{31}'a_{31}''a_{22}''a_{22}$ clearly form a rectangle, we thus have $d(a_{31}'',a_{31}')=d(a_{22}'',a_{22})=\frac{\sqrt{3}}{\sqrt{2}}$. 

For simplicity, denote $N:=d(a_{31}, a_{31}'')$. Then using the above found values, we must have $\tan(\frac{3}{2}\pi - \alpha')=\frac{\sqrt{3}/\sqrt{2}}{N}$, meaning $N=\frac{\sqrt{6}}{2}\cdot \cot(\frac{3}{2}\pi - \alpha')$. Thus, we have $d(a_{31}',a_{22})=d(a_{31}'',a_{22}'')=\frac{\sqrt{3}}{\sqrt{2}}-N=\frac{\sqrt{6}}{2}\cdot (1-\cot(\frac{3}{2}\pi - \alpha'))$. This is equation (5).
\end{proof}

\begin{proof}(of Lemma \ref{lemma:e_i_fermatpoint})    
    Without loss of generality, consider the case $i=2$. We need to show that the three angles of the triangle $c_2d_2d_1$ are less than $2\pi/3$. Since it will be used below, recall that $p$ denotes the center of the entire net.

    Due to symmetry, the triangle $c_2d_2d_1$ is isosceles. Additionally, the triangle $c_2d_2d_1$ lies inside the triangle $pd_2d_1$ which is also isosceles. It follows that $\angle d_2c_2d_1>\angle d_2pd_1=\pi/2$. Therefore, the two equal angles of the triangle $c_2d_2d_1$ are $\angle d_2d_1c_2=\angle c_2d_2d_1=\frac12(2\pi-\angle d_2c_2d_1)<\pi/4<2\pi/3$ as required.

    It remains to show that $\angle d_2c_2d_1<3\pi/2$. This will follow from the following chain of (in)equalities:
    \begin{align*}
        \angle d_2c_2d_1\stackrel{(i)}{=}\angle a_{21}c_2a_{12}\stackrel{(ii)}{<}\angle a_{21}b_2a_{12}\stackrel{(iii)}{=}2\pi/3
    \end{align*}
    
    (i) is true since these two angles are opposite angles (remember that $c_2$ was defined as the intersection of the line segments $\overline{d_2a_{12}}$ and $\overline{d_1a_{21}}$).

    (ii) is true because $c_2$ lies on the line through $p$ and $b_2$ (by symmetry) and $d(c_2,p)>d(b_2,p)$. This inequality is due to the following: $\angle c_{2}a_{21}a_{22}=\alpha<\frac{13\pi}{12}$ by Lemma \ref{lemma:alpha}, and $\angle b_2a_{21}a_{22}=13\pi/12$ by construction. Therefore, $c_2$ lies outside the dodecagon, further away from the center $p$ than $b_2$.

    (iii) is true since we initially fixed this interior angle of the dodecagon to exactly this value.
\end{proof}

\bibliographystyle{amsalpha}
\bibliography{references}

\end{document}